\documentclass[11pt,reqno]{amsart}
\usepackage[utf8]{inputenc}
\usepackage{amsopn,amssymb,mathrsfs,xcolor,url,enumitem,accents}
\usepackage[abbrev]{amsrefs} 
\usepackage[a4paper,lmargin=3cm,rmargin=3cm,tmargin=4cm,bmargin=4cm,marginparwidth=2.8cm,marginparsep=1mm]{geometry} 
\usepackage[bookmarks=true,hyperindex,pdftex,colorlinks,citecolor=blue]{hyperref} 
\usepackage{cleveref}

\makeatletter
\@namedef{subjclassname@2020}{\textup{2020} Mathematics Subject Classification}
\makeatother

\newtheorem{theorem}{Theorem}[section]

\newtheorem{lemma}[theorem]{Lemma}
\newtheorem{proposition}[theorem]{Proposition}
\newtheorem{problem}[theorem]{Problem}

\theoremstyle{definition}

\newtheorem{remark}[theorem]{Remark}

\newcommand{\cl}[1]{\ensuremath{\overline{{#1}}}}
\newcommand{\da}[2]{\ensuremath{\langle{#1},{#2}\rangle}}
\newcommand{\DEL}[2]{\ensuremath{\mathbf{\Delta}^{#1}_{#2}}}
\newcommand{\ep}{\varepsilon}
\newcommand{\free}[1]{\ensuremath{\mathcal{F}({#1})}}
\newcommand{\gur}{\mathbb{G}}
\newcommand{\IF}{\ensuremath{\mathbf{IF}}}
\newcommand{\map}[3]{\ensuremath{{#1}:{#2}\to{#3}}}
\newcommand{\n}[1]{\ensuremath{\left\|{#1}\right\|}}
\newcommand{\ndot}{\ensuremath{\left\|\cdot\right\|}}
\newcommand{\PI}[2]{\ensuremath{\mathbf{\Pi}^{#1}_{#2}}}
\newcommand{\pn}[2]{\ensuremath{\left\|{#1}\right\|_{#2}}}

\newcommand{\prsb}{\mathcal{B}}
\newcommand{\prsbap}{\mathcal{B}_{\mathrm{AP}}}
\newcommand{\Q}{\mathbb{Q}}
\newcommand{\R}{\mathbb{R}}
\newcommand{\restrict}[1]{\ensuremath{\mathord{\upharpoonright}_{#1}}}
\newcommand{\se}{\mathcal{SB}}
\newcommand{\set}[2]{\ensuremath{\left\{{#1}\;:\;\,{#2}\right\}}}
\newcommand{\SIG}[2]{\ensuremath{\mathbf{\Sigma}^{#1}_{#2}}}
\newcommand{\Tr}{\mathbf{Tr}}
\newcommand{\ury}{\mathbb{U}}
\newcommand{\vs}{\mathcal{R}}
\newcommand{\vsc}{\vs_{\mathrm{cpt}}}
\newcommand{\vsd}{\vs_{\mathrm{dis}}}
\newcommand{\vspu}{\vs_{\mathrm{p1u}}}
\newcommand{\vsap}{\vs_{\mathrm{AP}}}
\newcommand{\vsbap}{\vs_{\mathrm{BAP}}}
\newcommand{\WF}{\ensuremath{\mathbf{WF}}}

\DeclareMathOperator{\aspan}{span}
\DeclareMathOperator{\dom}{dom}
\DeclareMathOperator{\Lip}{Lip}

\renewcommand{\geq}{\geqslant}
\renewcommand{\leq}{\leqslant}

\numberwithin{equation}{section}

\allowdisplaybreaks

\begin{document}

\title[Lipschitz-free spaces and Bossard's reduction argument]{Lipschitz-free spaces and Bossard's reduction argument}

\date{\today}

\author[R. J. Smith]{Richard J. Smith}
\address[R. J. Smith]{School of Mathematics and Statistics, University College Dublin, Belfield, Dublin 4, Ireland}
\email{richard.smith@maths.ucd.ie}

\begin{abstract}
We set up a descriptive set-theoretic framework to study Lipschitz-free spaces and use the reduction argument of Bossard to prove several results. We prove two universality results: if a separable Banach space is isomorphically universal for the class of Lipschitz-free spaces over the countable complete discrete metric spaces then it is isomorphically universal for the class of separable Banach spaces, and if a complete separable metric space is Lipschitz universal for the same class of metric spaces then it is Lipschitz universal for all separable metric spaces. We also show that there exist countable complete discrete metric spaces whose Lipschitz-free spaces fail the bounded approximation property and are thus not isomorphic to any dual Banach space. Finally, we calculate the descriptive complexity of the classes of separable Banach spaces and separable Lipschitz-free spaces having the approximation property.
\end{abstract}
	
\maketitle

\section{Introduction}\label{sect_intro}

Let $M$ be a metric space with a base point denoted by $0$. Let $\Lip_0(M)$ denote the Banach space of real-valued Lipschitz functions that vanish at $0$, endowed with the norm that assigns to each $f \in \Lip_0(M)$ its optimal Lipschitz constant. Recall that the map $\map{\delta}{M}{\Lip_0(M)^*}$ given by $\da{f}{\delta(x)}=f(x)$, $x \in M$, $f \in \Lip_0(M)$, is an isometric embedding, and that we can define the Lipschitz-free space (hereafter free space) $\free{M}$ over $M$ as the closed linear span of the image $\delta(M) \subseteq \Lip_0(M)^*$ (where $\delta(M\setminus\{0\})$ is linearly independent).

For essential information on the fundamental properties of free spaces, we refer the reader to \cites{godefroy:kalton:03,weaver:18}; in the latter they are known as Arens-Eells spaces. Free spaces have close ties to metric geometry, optimal transport theory and the linear and non-linear theory of Banach spaces, and have been the subject of intense study by the functional analysis community in the last two decades. Despite the ease with which they can be defined and the considerable progress made in understanding them over the intervening years, their structure is still relatively poorly understood and lots of ``elementary'' questions about them persist unanswered.

One of the enduring themes of this research has been to find conditions on metric spaces to decide whether their corresponding free spaces are linearly isomorphic or not. If $M$ is Lipschitz isomorphic to a subspace of $N$ then $\free{M}$ is linearly isomorphic to a subspace of $\free{N}$, and if $M$ and $N$ are Lipschitz isomorphic then $\free{M}$ and $\free{N}$ are linearly isomorphic. However, the converses to these statements fail dramatically. In particular, it is not easy to find families of metric spaces whose corresponding free spaces are pairwise non-isomorphic.

The first example of an uncountable family of separable metric spaces whose free spaces are pairwise non-isomorphic was given in \cite{hajek:lancien:pernecka:16}. These metric spaces are all homeomorphic to the Cantor set $2^\omega$. Two more uncountable families of this kind are given in \cites{basset:25,basset:lancien:prochazka:25}. In \cite{basset:25}, each metric space in the family is countable and complete, while in \cite{basset:lancien:prochazka:25} each one is countable complete and discrete. The free spaces corresponding to the first family are non-isomorphic because the weak-fragmentability (or weak Szlenk) index of each one is distinct (and countable). The same holds in the case of the second family, this time with respect to the dentability index (not to be confused with the $w^*$-dentability index).

These results have consequences for universality. Recall that a metric space is purely 1-unrectifiable (p1u) if and only if it admits no Lipschitz isomorphic copy of any compact subset of $\R$ having strictly positive Lebesgue measure \cite{kirchheim:94}. By \cite{aliaga:gartand:petitjean:prochazka:22}, a complete separable metric space $M$ is p1u if and only if $\free{M}$ has the Radon-Nikod\'ym property (RNP); equivalently the dentability index of $\free{M}$ is countable. Because the dentability index of each free space associated with the second family is distinct, the set of indices must be unbounded in $\omega_1$. Thus, if a complete separable metric space has a free space that is isomorphically universal for the class of free spaces over countable complete discrete metric spaces, it cannot be p1u. Consequently, if a complete separable metric space is Lipschitz universal for the class of countable complete discrete metric spaces, it cannot be p1u \cite{basset:lancien:prochazka:25}*{Corollary 4.7}.

In this paper we use a tool from descriptive set theory, specifically Bossard's reduction argument (see \cite{dodos:10}*{p.~31}), to build on this universality result.

\begin{theorem}\label{thm:universal_1}
Let $X$ be a separable Banach space that is isomorphically universal for the class of free spaces over countable complete discrete metric spaces. Then $X$ is isomorphically universal for the class of separable Banach spaces. 
\end{theorem}

\begin{theorem}\label{thm:universal_2}
Let $M$ be a complete separable metric space that is Lipschitz universal for the class of countable complete discrete metric spaces. Then $M$ is Lipschitz universal for the class of separable metric spaces. 
\end{theorem}

Another enduring strand of this research has been to identify which free spaces have or fail one of Grothendieck's classical approximation properties. Recall that a Banach space $X$ is said to have the approximation property (AP) if, given compact $K \subseteq X$ and $\ep>0$, there exists a bounded linear finite-rank operator $\map{T}{X}{X}$ such that $\n{Tx-x} \leq \ep$ whenever $x \in K$. If there exists $\lambda \geq 1$ such that the operator $T$ can always be chosen to satisfy $\n{T} \leq \lambda$, then $X$ is said to have the $\lambda$-bounded approximation property ($\lambda$-BAP). We say that $X$ has the bounded approximation property (BAP) if it has the $\lambda$-BAP for some $\lambda$, and the metric approximation property (MAP) if it has the $1$-BAP.

The literature on free spaces and approximation properties is too abundant to be summarised in this short paper; we invite the interested reader to consult \cite{godefroy:24} and the introduction of \cite{smith:talimdjioski:23} and references therein for more information. We only need to recall two results for our purposes. First, it follows immediately from \cite{godefroy:kalton:03}*{Theorem 3.1} that if a separable Banach space $X$ fails the AP then so does $\free{X}$. Second, \cite{aliaga:nous:petitjean:prochazka:21}*{Corollary 2.8} implies that $\free{M}$ has the AP whenever $M$ is countable discrete and complete. To the best of my knowledge, to date there has been no example of a complete separable p1u metric space whose free space has the AP but fails the BAP. The next result shows that there are such examples and, moreover, they can be chosen to be countable complete and discrete. Again, this result is based on Bossard's argument. Another crucial ingredient is the fact that, relative to the established descriptive set-theoretic frameworks for separable Banach spaces (see \Cref{sec:prelims}), the class of separable Banach spaces that satisfies the BAP is Borel (though being analytic would have sufficed) \cite{ghawadrah:17}. 

\begin{theorem}\label{thm:fail_BAP_simple}
There exist countable complete discrete metric spaces $M$ such that $\free{M}$ fails the BAP. Consequently, $\free{M}$ is not isomorphic to a dual space.
\end{theorem}

This result can be contrasted with \cite{aliaga:gartand:petitjean:prochazka:22}*{Theorem 3.2} which states, among other things, that $\free{M}$ is isometric to a dual space whenever $M$ is a p1u metric space that is proper, i.e.~its closed balls are compact. It can also be regarded in the light of Kalton's problem of whether $\free{M}$ has the BAP whenever $M$ is uniformly discrete; see \cite{kalton:04}*{p.~185} and e.g.~\cite{godefroy:24}*{Problem 4.1} for more context. In this vein, there exists a uniformly discrete space whose free space is not isometric to a dual space  (though this particular example has the MAP) \cite{garcia-lirola:petitjean:prochazka:ruedazoca:18}*{Example 5.8}.

We prove the results above using \Cref{lem:main_tool}, which is one of our two main tools. The other main tool, \Cref{lem:combine}, will be needed for our final main result,  \Cref{thm:ap_coanalytic}, which shows that the classes of separable Banach spaces and separable free spaces having the AP are, relative to the appropriate descriptive set-theoretic interpretations, complete coanalytic sets. (We remark that this is not the first time Polish spaces of metrics have been used to study free spaces and approximation properties; see \cites{talimdjioski:23,smith:talimdjioski:24,talimdjioski:25}.)

We prove these theorems in Sections \ref{sec:results} and \ref{sect:ap}, and make some related observations along the way. The application of descriptive set theory to Banach space theory can be traced to some seminal papers of Bourgain, e.g.~\cite{bourgain:80}. It was formalised later by Bossard in another seminal paper \cite{bossard:02} and has since developed into a rich and beautiful subfield that boasts many powerful techniques and results, see e.g.~\cite{dodos:10} and references therein, and \cite{cuth:doucha:kurka:22b} and related papers for more recent developments. I speculate that techniques from descriptive set theory could be used to prove more results about free spaces in the future.  

I heard about the results in \cite{basset:lancien:prochazka:25} in a talk given by E.~Basset at the Banff International Research Station workshop on Functional and Metric Analysis and their Interactions, held at the Institute of Mathematics, University of Granada, Spain, in May 2025. It occurred to me that it might be possible to use descriptive set theory, specifically Bossard's reduction argument (see \cite{dodos:10}*{p.~31}), to build on their universality results. I would like to thank E.~Basset, G.~Lancien and A.~Proch\'azka for the fascinating exchanges that ensued. In particular, \Cref{thm:universal_2} was prompted by a question asked by G. Lancien. I would also like to thank M.~C\'uth, G.~Godefroy and A.~Rueda Zoca for several remarks which helped me when I was preparing this paper. Finally, I should stress that I have left out most of the results in \cites{basset:25,basset:lancien:prochazka:25} from this introduction purely in the interests of brevity, and I recommend that the interested reader should consult these papers in full.

\section{Notation and preliminaries}\label{sec:prelims}

We use standard Banach space notation throughout. The closed unit ball of a Banach space $X$ is denoted $B_X$. Given Banach spaces $X$ and $Y$, we write $X \hookrightarrow Y$ and $X \equiv Y$ when $X$ is linearly isomorphic to a subspace of $Y$ and $X$ is linearly isometric to $Y$, respectively. We also use these symbols in the context of metric spaces $M$ and $N$, with ``Lipschitz'' substituted for ``linear''. By \cite{weaver:18}*{Theorems 3.6 and 3.7}, if $M \hookrightarrow N$ (in the Lipschitz sense) then $\free{M}\hookrightarrow \free{N}$ (in the linear sense). We say that the Banach space $X$ is isomorphically universal for a given class $\mathcal{C}$ of Banach spaces if $Y \hookrightarrow X$ whenever $Y \in \mathcal{C}$; Lipschitz universality of metric spaces is defined analogously. 

For concepts pertaining to descriptive set theory, we follow \cite{kechris:95}. The sets of natural numbers, countable ordinals and infinite sequences of natural numbers are denoted $\omega$, $\omega_1$ and $\omega^\omega$, respectively. We identify the power set of a set $\Gamma$ with $2^\Gamma$. We denote by $\SIG{0}{\alpha}$ and $\PI{0}{\alpha}$ the additive and multiplicative Borel classes of order $\alpha<\omega_1$ in a Polish space, so that $F_\sigma$, $G_\delta$ and $F_{\sigma\delta}$-sets are $\SIG{0}{2}$, $\PI{0}{2}$ and $\PI{0}{3}$-sets, and so on. Analytic, coanalytic and Borel subsets of Polish spaces are called $\SIG{1}{1}$, $\PI{1}{1}$ and $\DEL{1}{1}$-sets, respectively. Given subsets $A \subseteq X$, $B \subseteq Y$ of Polish spaces $X$ and $Y$, we say that $A$ is reducible to $B$ if there is a continuous function $\map{f}{X}{Y}$ such that $A=f^{-1}(B)$. If $\mathbf{\Gamma}(X)$ is a class of subsets of a Polish space $X$, a subset $B \subseteq Y$ of a Polish space $Y$ is called $\mathbf{\Gamma}$-hard if $A$ is reducible to $B$ whenever $A \in \mathbf{\Gamma}(X)$ and $X$ is Polish and zero-dimensional, and $\mathbf{\Gamma}$-complete if $B$ is $\mathbf{\Gamma}$-hard and in addition belongs to $\mathbf{\Gamma}(Y)$.

The tree of finite (possibly empty) sequences of natural numbers is denoted $\omega^{<\omega}$. Given $s,t \in \omega^{<\omega}$ and $n\in\omega$, we define $|s|=\dom s \in \omega$ and $s^\frown n = s \cup \{(|s|,n)\}$, we write $s \preccurlyeq t$ when $t$ is an extension of $s$, i.e.~$s=t\restrict{|s|}$, and $s \perp t$ when $s,t$ are incomparable, and finally we let $s \wedge t$ denote the greatest element of $\omega^{<\omega}$ less than both $s$ and $t$. The space of trees, i.e.~downwards-closed subsets of $\omega^{<\omega}$, is denoted $\Tr$. With respect to the pointwise topology this is closed in $2^{\omega^{<\omega}}$ and hence compact. Within $\Tr$ we consider the subsets $\WF$ and $\IF$ of well-founded and ill-founded trees, which are well-known to be $\PI{1}{1}$ and $\SIG{1}{1}$-complete, respectively. 

Next, we set up an appropriate descriptive set-theoretic framework for the class of separable infinite-dimensional free spaces. We shall consider the Polish space $[0,\infty)^{\omega\times\omega}$ endowed with the pointwise topology, and the set $\vs \subseteq [0,\infty)^{\omega\times\omega}$ of metrics on $\omega$. This space was introduced by Vershik in \cite{vershik:98} and features again in e.g.~\cite{vershik:04} and \cite{clemens:12} (in these works $\vs$ is defined instead as the set of all pseudo-metrics on $\omega$, that is, functions $d\in[0,\infty)^{\omega\times\omega}$ that are metrics except that it is possible to have distinct $i,j \in \omega$ satisfying $d(i,j)=0$). It is easily seen that $\vs$ is a $G_\delta$ in $[0,\infty)^{\omega\times\omega}$ and is thus Polish as well. This space can be used to code the class of complete separable infinite metric spaces. Given such a space $(M,\rho)$, fix a dense sequence of distinct points $(x_i) \subseteq M$ and consider $d \in \vs$ given by $d(i,j)=\rho(x_i,x_j)$, $i,j \in \omega$; then $(M,\rho) \equiv M_d$, where $M_d$ is the completion of $(\omega,d)$. Conversely any such completion is separable and infinite. Because a free space is separable and infinite-dimensional if and only if its underlying metric space is separable and infinite, and because the free space over a metric space is isometric to that of its completion, $\vs$ can be seen as a natural coding of the class of separable infinite-dimensional free spaces. As a change of base point does not alter the isometric structure of the corresponding free space, we shall always regard $0\in\omega$ as the base point. 

Bossard coded the class of separable infinite-dimensional Banach spaces as a $\DEL{1}{1}$ subset $\se$ of the set of closed subsets of $C(2^\omega)$ equipped with the Effros-Borel structure. Following a programme initiated in \cite{godefroy:saint-raymond:18}, the authors of \cites{cuth:doucha:kurka:22a,cuth:doucha:kurka:22b} (and related papers) observed that it is possible to code the class of separable infinite-dimensional Banach spaces using an arguably canonical Polish space $\prsb$. First they let $V$ be the countable $\Q$-linear vector space of elements of $c_{00}$ having rational entries, and then they define $\prsb$ to be the set of all functions $\mu \in [0,\infty)^V$ which can be extended (uniquely) to a norm on $c_{00}$. Endowed with the pointwise topology, $\prsb$ is a $G_\delta$ in $[0,\infty)^V$ and is thus Polish. Every separable and infinite-dimensional Banach space $(X,\ndot)$ has representatives in $\prsb$. Given a linearly independent sequence $(x_i) \subseteq X$ having dense span, define $\mu(v)=\n{\sum_{i \in \omega} v(i)x_i}$, where $v \in V$. Then $X \equiv X_\mu$, where $X_\mu$ is the completion of $(V,\mu)$; conversely, every such completion is a separable and infinite-dimensional Banach space. We will work with $\prsb$ instead of $\se$ because, in our opinion, it is easier to use and more compatible with $\vs$. While not directly relevant to this paper, another advantage of using $\prsb$ is that Baire category arguments can be applied, e.g.~the set of $\mu \in \prsb$ such that $X_\mu \equiv \gur$, where $\gur$ is the Gurari{\u\i} space, is a dense $G_\delta$ \cite{cuth:doucha:kurka:22b}*{Theorem 4.1}, hence $\gur$ can be thought of as the generic separable Banach space. According to \cite{cuth:doucha:kurka:22a}*{Theorem 8}, there is a Borel isomorphism $\map{\Theta}{\prsb}{\se}$ such that $\Theta(\mu) \equiv X_\mu$ for all $\mu \in \prsb$. It follows easily from this and \cite{bossard:02}*{Theorem 2.3 (ii)} that, given any separable Banach space $X$, the set $\langle{X}\rangle^{\prsb}_{\hookrightarrow}:=\set{\mu \in \prsb}{X_\mu \hookrightarrow X}$ is $\SIG{1}{1}$.

Concerning Polish versus Effros-Borel structures, we should mention that it is possible to code the class of complete infinite separable metric spaces using the set of closed infinite subsets of the Urysohn universal separable metric space $\ury$, again equipped with the Effros-Borel structure, see e.g.~\cites{gao:09,gao:kechris:03}. However we shall work with Vershik's space $\vs$ for the reasons given above. Also, analogously to $\gur$, the set of $d\in\vs$ such that $M_d \equiv \ury$ is a dense $G_\delta$ \cite{vershik:04}*{Theorems 1 and 3}. Thus $\free{\ury}$, called the Holmes space in \cite{fonf:wojtaszczyk:08}, could be regarded as the generic separable free space. An immediate consequence of this and \cite{fonf:wojtaszczyk:08}*{Theorem 2.1} is that the set of $d \in \vs$ for which $\free{M_d}$ has the metric approximation property is residual. This observation complements \Cref{thm:ap_coanalytic}.

\section{Three main results}\label{sec:results}

Our first result will allow us to apply to $\vs$ one of the observations about $\prsb$ recorded above.

\begin{proposition}\label{pr:reduction}
There is a continuous map $\map{\Xi}{\vs}{\prsb}$ such that $X_{\Xi(d)}\equiv \free{M_d}$ for all $d\in\vs$.
\end{proposition}

\begin{proof}
Recall that, regardless of $d\in \vs$, the sequence $(\delta(i+1))_{i\in\omega}$ is always linearly independent and has dense span in $\free{M_d}$. Thus if we define $\Xi:\vs\to\prsb$ by
\[
\Xi(d)(v) = \n{\sum_{i\in\omega} v(i)\delta(i+1)}_{\free{M_d}}, \quad v \in V,
\]
we quickly conclude that $X_{\Xi(d)}\equiv\free{M_d}$.

It remains to show that $\Xi$ is continuous. Let $d \in \vs$,  $v_1,\ldots,v_k \in V$ and $\ep \in (0,1)$. To establish the continuity of $\Xi$, we shall find $N\in\omega$ and $\delta > 0$ such that
\[
|\Xi(d')(v_\ell) - \Xi(d)(v_\ell)| < \ep, \quad 1 \leq \ell \leq k,
\]
whenever $d' \in \vs$ satisfies $|d'(i,j)-d(i,j)|<\delta$ for all $i,j \leq N$. 

To begin, we pick $N \in \omega$ such that $v_\ell(i)=0$ whenever $1 \leq \ell \leq k$ and $i \geq N$. Then set
\[
r = \min\set{d(i,j)}{i < j \leq N} > 0, \quad L = 1 + \max\set{\Xi(d)(v_\ell)}{1 \leq \ell \leq k},
\]
$\alpha=\ep/(L+1)$ and $\delta = \alpha r$.

Now consider any $d' \in \vs$ such that $|d'(i,j)-d(i,j)|<\delta$ whenever $i,j \leq N$. Let $f \in B_{\Lip_0(M_{d'})}$. Given that
\begin{equation}\label{eqn:reduction}
(1-\alpha)d(i,j) \leq d(i,j)-\delta < d'(i,j) < d(i,j) + \delta \leq (1+\alpha)d(i,j), \quad i< j \leq N,
\end{equation}
the restriction of $f$ to $N+1$ is $(1+\alpha)$-$d$-Lipschitz. By McShane's extension theorem \cite{weaver:18}*{Theorem 1.33}, there exists $g \in (1+\alpha)B_{\Lip_0(M_d)}$ such that $g(i)=f(i)$ whenever $i \leq N$. It follows that
\begin{align*}
\sum_{i\in\omega} v_\ell(i)f(i+1) = \sum_{i<N} v_\ell(i)f(i+1) &= \sum_{i<N} v_\ell(i)g(i+1)\\
&= \sum_{i\in\omega} v_\ell(i)g(i+1) \leq (1+\alpha)\Xi(d)(v_\ell), \quad 1 \leq \ell \leq k.
\end{align*}
As these inequalities hold for all such $f$, we obtain
\[
\Xi(d')(v_\ell) \leq (1+\alpha)\Xi(d)(v_\ell) < \Xi(d)(v_\ell) + \ep < L, \quad 1 \leq \ell \leq k.
\]

Now we swap $d$ and $d'$ and largely repeat the above. Given $f \in B_{\Lip_0(M_d)}$, by \eqref{eqn:reduction} and McShane's theorem, we can find $g \in (1-\alpha)^{-1}B_{\Lip_0(M_{d'})}$ such that $g(i)=f(i)$ whenever $i \leq N$. It follows that
\[
\Xi(d)(v_\ell) \leq \frac{1}{1-\alpha}\Xi(d')(v_\ell) = \bigg(1+\frac{\ep}{L}\bigg)\Xi(d')(v_\ell) < \Xi(d')(v_\ell)+\ep, \quad 1 \leq \ell \leq k.
\]
This completes the proof.
\end{proof}

\begin{proposition}\label{co:analytic}
Given a separable Banach space $X$ and a complete separable metric space $M$, the sets
\[
\langle X\rangle^{\vs}_{\hookrightarrow} := \set{d \in \vs}{\free{M_d} \hookrightarrow X}\quad\text{and}\quad \langle M\rangle^{\vs}_{\hookrightarrow} := \set{d \in \vs}{M_d \hookrightarrow M}
\]
are $\SIG{1}{1}$.
\end{proposition}

\begin{proof}
Evidently $\langle X\rangle^{\vs}_{\hookrightarrow} = \Xi^{-1}(\langle{X}\rangle^{\prsb}_{\hookrightarrow})$ is $\SIG{1}{1}$. Similarly to \cite{clemens:12}*{Lemma 4}, the set
\[
\set{(d,d') \in \vs\times\vs}{M_d \hookrightarrow M_{d'}}
\]
is $\SIG{1}{1}$, whence $\langle M\rangle^{\vs}_{\hookrightarrow}$ is $\SIG{1}{1}$.
\end{proof}

Next, we determine the descriptive complexity of the classes of countably infinite complete discrete metric spaces and complete separable p1u metric spaces.

\begin{proposition}\label{pr:coanalytic}
The set $\vsd := \set{d \in \vs}{M_d \text{ is discrete}}$ is $\PI{1}{1}$.
\end{proposition}

\begin{proof}
It is clear that $d \in \vsd$ if and only if every Cauchy sequence in $(\omega,d)$ is eventually constant. Given $k,m \in \omega$, define the sets
\[
 A_{k,m} = \set{(d,(x_i)) \in \vs \times \omega^\omega}{d(x_i,x_j) \leq 2^{-k} \text{ for all }i,j \geq m},
\]
and
\[
 B_m = \set{(x_i) \in \omega^\omega}{x_i=x_m \text{ whenever }i \geq m},
\]
which are closed in $\vs \times \omega^\omega$ and $\omega^\omega$, respectively. Then $(d,(x_i)) \in A:=\bigcap_{k \in \omega}\bigcup_{m\in \omega} A_{k,m}$ if and only if $(x_i) \in \omega^\omega$ is Cauchy in $(\omega,d)$, and $(x_i) \in B:=\bigcup_{m\in\omega} B_m$ if and only if $(x_i)$ is eventually constant. Therefore $\vs\setminus\vsd$ is $\SIG{1}{1}$ because it is the image of $A\setminus (\vs\times B) \subseteq \vs\times\omega^\omega$ under the projection onto the first coordinate.
\end{proof}

\begin{proposition}\label{pr:p1u_coanalytic}
The set $\vspu := \set{d \in \vs}{M_d \text{ is p1u}}$ is $\PI{1}{1}$.
\end{proposition}

\begin{proof}
Using the Borel isomorphism $\map{\Theta}{\prsb}{\se}$ and \cite{bossard:02}*{Corollary 3.3}, the set of $\mu \in \prsb$ such that $X_\mu$ has the RNP is $\PI{1}{1}$. Now apply \Cref{pr:reduction} and \cite{aliaga:gartand:petitjean:prochazka:22}*{Theorem 4.6}.
\end{proof}

\Cref{pr:p1u_coanalytic} can be proved more directly using Kirchheim's characterisation of p1u spaces mentioned above. Later, in \Cref{pr:complete}, we upgrade $\vsd$ and $\vspu$ to $\PI{1}{1}$-complete. For this and the main results we will need the following lemma, which is a tool for generating continuous maps from power sets to $\vs$. 

\begin{lemma}\label{lm:continuous_metric}
Let $(M,d)$ be a metric space, $A \subseteq \omega$, $\map{\phi_0}{A}{M}$ and $\map{\phi_1,\phi_2}{\omega\setminus A}{M}$ injections with pairwise disjoint ranges, and let $\map{\tau}{\omega\setminus A}{\Gamma}$, where $\Gamma$ is a set. Then the map $\map{\Psi}{2^\Gamma}{M^\omega}$ given by
\[
\Psi(E)(n) = \begin{cases} \phi_0(n) & \text{if }n \in A,\\
\phi_1(n) & \text{if }n \in \omega\setminus A \text{ and }\tau(n) \notin E,\\
\phi_2(n) & \text{if }n \in \omega\setminus A \text{ and }\tau(n) \in E, \end{cases} \quad E \in 2^\Gamma,
\]
is continuous. Moreover, the map $\map{\Delta}{2^\Gamma}{[0,\infty)^{\omega\times\omega}}$ given by
\[
\Delta(E)(m,n)=d(\Psi(E)(m),\Psi(E)(n)), \quad m,n \in \omega,
\]
is continuous and takes values in $\vs$. Finally, $(\omega,\Delta(E))$ is isometric to $(\Psi(E)(\omega),d)$.
\end{lemma}

\begin{proof}
To check continuity of $\Psi$, let $E \in 2^\Gamma$ and let $F \subseteq \omega$ be finite. Define the finite sets $G=\tau(F\setminus A) \cap E$ and $H=\tau(F\setminus A)\setminus E$; then $G\subseteq E$, $H \cap E=\varnothing$ and $\Psi(E')(n)=\Psi(E)(n)$ whenever $n \in F$, $E' \in 2^\Gamma$, $G\subseteq E'$ and $H \cap E' =\varnothing$. Therefore $\Psi$ is continuous. As $\Delta$ is the composition of $\Psi$ and another continuous map, it is also continuous. Since $\phi_0,\phi_1,\phi_2$ are injective and have pairwise disjoint ranges, $\Psi(E)$ is injective for all $E \in 2^\Gamma$. It follows that $\Delta(E) \in \vs$. The last assertion is obvious.
\end{proof}

\begin{remark}
It is a simple exercise to show that $\vsc := \set{d \in \vs}{M_d \text{ is compact}}$ is $\PI{0}{3}$. As a straightforward application, we can show that $\vsc$ is $\PI{0}{3}$-complete using \Cref{lm:continuous_metric}. Indeed, let
\[
P = \set{E \in 2^{\omega\times\omega}}{\set{j \in \omega}{(i,j) \in E} \text{ is finite for all }i \in \omega}.
\]
According to \cite{kechris:95}*{Exercise 23.1}, $P$ is $\PI{0}{3}$-complete. Now set $\Gamma=\omega\times\omega$, $M=\ell_1(\Gamma \cup \{\infty\})$, $A=\{0\}$ and $\phi_0=\{(0,0)\}$. Let $d$ be the usual metric on $M$, let $\map{\tau}{\omega\setminus A}{\Gamma}$ be a bijection, let $(e_\gamma)_{\gamma \in \Gamma\cup\{\infty\}}$ be the usual basis, and define injections $\map{\phi_1,\phi_2}{\omega\setminus A}{M}$ by
\[
\phi_1(n) = 2^{-n}e_\infty \quad\text{and}\quad \phi_2(n) = 2^{-p(\tau(n))}e_{\tau(n)},
\]
where $\map{p}{\omega\times\omega}{\omega}$ is the projection onto the first coordinate. Evidently the ranges of $\phi_0,\phi_1,\phi_2$ are pairwise disjoint. Let $\Delta$ be the map furnished by \Cref{lm:continuous_metric}. Then it is straightforward to check that $M_{\Delta(E)} = (\omega,\Delta(E))$ always, and this space is compact whenever $E \in P$ and not compact otherwise. Therefore $\Delta$ is a continuous reduction of $P$ to $\vsc$, which implies that the latter is $\PI{0}{3}$-complete as claimed. The precise complexity of a few other standard classes of complete separable metric spaces can be determined in a similar manner.
\end{remark}

We will be interested in metrics $\rho$ on $\omega^{<\omega}$ satisfying the following property.
\begin{align}
 &\text{Given $\ep>0$ and an infinite branch $B \subseteq \omega^{<\omega}$, there exists $n\in\omega$ such that} \nonumber\\
 &\text{$\rho(s,t) < \ep$ whenever $s,t \in B$ and $|s|,|t| \geq n$.} \label{eqn:metric}
\end{align}

In the next result and a number of others, $\cl{M}$ shall denote the completion of the metric space $M$.

\begin{lemma}\label{lm:containing_M}
Let $M=(M,\theta)$ be a complete separable metric space, let $(x_i)_{i\in\omega}$ be a dense sequence in $M$ whose elements repeat infinitely often, let $\rho$ be a metric on $\omega^{<\omega}$ satisfying \eqref{eqn:metric} and define a second metric $\eta$ on $\omega^{<\omega}$ by
 \[
  \eta(s,t) = \rho(s,t)+\theta(x_{|s|},x_{|t|}).
 \]
Then whenever $T \in \IF$ and $B \subseteq T$ is an infinite branch, there exists an isometric embedding $\map{\phi}{M}{\cl{(B,\eta)} \subseteq \cl{(T,\eta)}}$. Moreover, if $M$ is complete then there is a 1-Lipschitz retraction $\map{r}{\cl{(T,\eta)}}{\phi(M)}$.
\end{lemma}

\begin{proof}
Let $q \in \omega^{\omega}$ such that $B=\set{q\restrict{n}}{n \in \omega}$. We define an isometric embedding $\map{\phi}{M}{\cl{(B,\eta)}}$ as follows. Given $x \in M$, because the elements of $(x_i)$ repeat infinitely often, there exist $i_0 < i_1 < i_2 < \ldots$ such that $\theta(x_{i_k},x)\to 0$. Given that $\rho$ satisfies \eqref{eqn:metric}, $(q\restrict{i_k})$ is $\rho$-Cauchy, hence $(q\restrict{i_k})$ is also $\eta$-Cauchy. Let $\phi(x) = \lim_k q\restrict{i_k} \in \cl{(B,\eta)}$. This is well-defined because if $j_0 < j_1 < j_2 < \ldots$ also satisfy $\theta(x_{j_k},x)\to 0$ then $\eta(q\restrict{i_k},q\restrict{j_k}) = \rho(q\restrict{i_k},q\restrict{j_k})+\theta(x_{i_k},x_{j_k}) \to 0$, again by \eqref{eqn:metric}. Next, we check that $\phi$ is an isometric embedding. With $x$ and the $i_k$ as above, let $y \in M$ and this time let $j_0 < j_1 < j_2 < \ldots$ such that $\theta(y,x_{j_k})\to 0$. Then once more by \eqref{eqn:metric},
\[
 \eta(\phi(x),\phi(y)) = \lim_k \eta(q\restrict{i_k},q\restrict{j_k}) = \lim_k \rho(q\restrict{i_k},q\restrict{j_k})+\theta(x_{i_k},x_{j_k}) = \theta(x,y).
\]
Finally, assume $M$ is complete. Define $\map{\zeta}{T}{M}$ by $\zeta(s)=x_{|s|}$. This map is $1$-Lipschitz because $\theta(\zeta(s),\zeta(t)) = \theta(x_{|s|},x_{|t|}) \leq \eta(s,t)$, $s,t \in T$. Let $\map{r}{\cl{(T,\eta)}}{\cl{(B,\eta)}}$ be the $1$-Lipschitz extension of $\phi \circ \zeta$ to $\cl{(T,\eta)}$. By the continuity of $r$ and the completeness of $M$, we have $r(\cl{(T,\eta)}) \subseteq \phi(M)$. Moreover $r(\phi(x))=\phi(x)$ whenever $x \in M$. Indeed, given $i_0 < i_1 < i_2 < \ldots$ such that $\theta(x_{i_k},x)\to 0$, we have $\phi(x)=\lim_k q\restrict{i_k}$, so $r(\phi(x)) = \lim_k\phi(\zeta(q\restrict{i_k}))=\lim_k \phi(x_{i_k})=\phi(x)$.
\end{proof}

If $(M,\theta)$ above is perfect then there is no need to insist that the elements of $(x_i)$ repeat infinitely often.

We now define a metric $\rho$ on $\omega^{<\omega}$ that we fix hereafter. Let $(e_s)_{s\in \omega^{<\omega}}$ denote the usual basis of $c_{00}(\omega^{<\omega})$ and define the injection $\map{\Omega}{\omega^{<\omega}}{c_{00}(\omega^{<\omega})}$ by
\[
 \Omega(s) = \sum_{r \preccurlyeq s} 2^{-|r|}e_r, \quad s \in \omega^{<\omega}.
\]
Then define $\rho(s,t)=\pn{\Omega(s)-\Omega(t)}{\infty}$, $s,t \in \omega^{<\omega}$. We see that $\rho$ satisfies \eqref{eqn:metric}. If $s \perp t$ then $s \wedge t \prec s,t$ and
\begin{equation}\label{eqn:metric_calc}
 \rho(s,t) = \pn{\sum_{s \wedge t \prec r \preccurlyeq s} 2^{-|r|}e_r - \sum_{s \wedge t \prec r \preccurlyeq t} 2^{-|r|}e_r}{\infty} = 2^{-|s \wedge t|-1},
\end{equation}
and we get the same expression for $\rho(s,t)$ if $s$ and $t$ are distinct and comparable. Hence \eqref{eqn:metric} holds.

\begin{lemma}\label{lm:discrete_complete}
Let $\eta$ be the metric from \Cref{lm:containing_M}, where $(M,\theta)$ is an arbitrary separable metric space and $\rho$ is as above. Then $(\omega^{<\omega},\eta)$ is discrete, and in addition $(T,\eta)$ is complete whenever $T \in \mathbf{WF}$.
\end{lemma}

\begin{proof}
Fix $s \in \omega^{<\omega}$. Given $t \in \omega^{<\omega}\setminus\{s\}$, by \eqref{eqn:metric_calc} we have 
\[
 \eta(s,t) \geq \rho(s,t) = 2^{-|s \wedge t|-1} \geq 2^{-|s|-1}.
\]
Hence $\eta(s,t) \geq 2^{-|s|-1}$ whenever $t \in \omega^{<\omega}\setminus\{s\}$, meaning that $\eta$ is discrete.

Now let $T \in \mathbf{WF}$ and let $(s_n) \subseteq T$ be $\eta$-Cauchy. We find a strictly increasing finite sequence $r_0 \prec r_1 \prec r_2 \prec \ldots$ of elements of $\omega^{<\omega}$ and infinite subsets $L_0 \supseteq L_1 \supseteq L_2 \supseteq \ldots$ of $\omega$ such that $r_k \preccurlyeq s_n$ for all $n \in L_k$. Begin by setting $r_0 = \varnothing$ and $L_0=\omega$. Given $r_k$ and $L_k$, either there exists $m \in \omega$ and an infinite subset $L_{k+1} \subseteq L_k$ such that $r_k^\frown m \preccurlyeq s_n$ for all $n \in L_{k+1}$, or not. If the first case holds, set $r_{k+1} = r_k^\frown m$ and continue the process, else terminate it. Because $T \in \mathbf{WF}$, this process must terminate at some finite stage $k$. By construction we know that $L_k$ is infinite and $r_k \preccurlyeq s_n$ for all $n \in L_k$ but also, given $m \in \omega$, the set $H_m:=\set{n \in L_k}{r_k^\frown m \preccurlyeq s_n}$ must be finite. If $n_1 \in H_{m_1}$ and $n_2 \in H_{m_2}$ for distinct $m_1,m_2 \in \omega$, then $s_{n_1} \perp s_{n_2}$ and $s_{n_1} \wedge s_{n_2} = r_k$, giving
\[
 \eta(s_{n_1},s_{n_2}) \geq \rho(s_{n_1},s_{n_2}) = 2^{-|r_k|-1}.
\]
Because $(s_n)$ is $\eta$-Cauchy, it follows that $H_m$ is empty for all but finitely many $m \in \omega$, and thus $s_n = r_k$ for infinitely many $n \in L_k$. Again because $(s_n)$ is $\eta$-Cauchy, we conclude that it must converge to $r_k$. This proves that $(T,\eta)$ is complete.
\end{proof}

We make use of the last three lemmas in the next result. \Cref{lem:combine,lem:main_tool} below are our main tools.

\begin{lemma}\label{lem:combine} Let $M=(M,\theta)$ be a complete separable metric space. There exists a continuous map $\map{\Delta}{\Tr}{\vs}$ such that $\Delta(T) \in \vsd$ whenever $T \in \WF$ and $M$ is isometric to a 1-Lipschitz retract of $M_{\Delta(T)}$ whenever $T \in \IF$.
\end{lemma}

\begin{proof} Let $(\omega^{<\omega},\eta)$ be the metric space obtained from \Cref{lm:containing_M} applied to $M$. Let $W=\omega^{<\omega} \cup (\omega^{<\omega} \times \{0\})$. We extend $\eta$ to a metric on $W$, again labelled $\eta$, by setting
 \[
  \eta(x,y) = \begin{cases} 0 & \text{if }x,y \notin \omega^{<\omega},\, x = y,\\
  1 & \text{if }x,y \notin \omega^{<\omega},\, x \neq y,\\
  1+\eta(\varnothing,x) & \text{if }x \in \omega^{<\omega},\, y \notin \omega^{<\omega},\\
  1+\eta(\varnothing,y) & \text{if }x \notin \omega^{<\omega},\, y \in \omega^{<\omega}.
  \end{cases}
 \]
It is routine to verify that this extension satisfies the triangle inequality.

Set $A=\phi_0=\varnothing$, $\Gamma=\omega^{<\omega}$, let $\map{\tau}{\omega}{\Gamma}$ be a bijection, and define injections $\map{\phi_1,\phi_2}{\omega}{W}$ by $\phi_1(n)=(\tau(n),0)$ and $\phi_2(n)=\tau(n)$, $n\in\omega$. Let $\map{\Delta}{2^{\Gamma}}{\vs}$ be the continuous map furnished by \Cref{lm:continuous_metric}, which we then restrict to $\Tr$. We observe that $\Psi(T)(\omega) = T \cup ((\omega^{<\omega}\setminus T) \times \{0\})$ for all $T \in \Tr$.

Let $T\in\WF$. Using \Cref{lm:discrete_complete}, $(T,\eta)$ is discrete and complete.  Hence $(\omega,\Delta(T))$, which is isometric to $(\Psi(T)(\omega),\eta)$, is also discrete and complete because every point in $(\omega^{<\omega}\setminus T) \times \{0\}$ is $\eta$-separated from everything else by at least $1$. Therefore $\Delta(T) \in \vsd$.

Now let $T\in\IF$. We claim that $M$ is isometric to a 1-Lipschitz retract of $M_{\Delta(T)}$. First of all, the function $\map{r'}{\Psi(T)(\omega)}{T}$, defined by
\[
 r'(x) = \begin{cases} x & \text{if }x \in T,\\
  \varnothing & \text{if }x \notin T,
  \end{cases}
\]
is a $1$-Lipschitz retraction of $(\Psi(T)(\omega),\eta)$ onto $(T,\eta)$. We extend this to a $1$-Lipschitz retraction $\map{r'}{\cl{(\Psi(T)(\omega),\eta)}}{\cl{(T,\eta)}}$. According to \Cref{lm:containing_M}, there is an isometric embedding $\map{\phi}{M}{\cl{(T,\eta)}}$ and a $1$-Lipschitz retraction $\map{r}{\cl{(T,\eta)}}{\phi(M)}$. Therefore $\map{r \circ r'}{\cl{(\Psi(T)(\omega),\eta)}}{\phi(M)}$ is a 1-Lipschitz retraction. Since $M_{\Delta(T)} \equiv \cl{(\Psi(T)(\omega),\eta)}$, our claim is proved.
\end{proof}

The final lemma below uses Bossard's reduction argument \cite{bossard:02}*{Theorem 3.2}; Theorems \ref{thm:universal_1} -- \ref{thm:fail_BAP_simple} emerge as relatively straightforward consequences of it.

\begin{lemma}\label{lem:main_tool} Let $\mathcal{A} \subseteq \mathcal{R}$ have the property that $d' \in \mathcal{A}$ whenever $d' \in \vs$ and $M_{d'}$ is isometric to a $1$-Lipschitz retract of $M_d$ for some $d \in \mathcal{A}$. Assume that $\vsd\cup\mathcal{A}$ is $\SIG{1}{1}$. Then $\vsd \cup \mathcal{A}=\mathcal{R}$.
\end{lemma}

\begin{proof}
Let $d \in \mathcal{R}$. Apply \Cref{lem:combine} to $M:=M_d$. The assumption $\vsd\cup\mathcal{A}$ is $\SIG{1}{1}$ implies $\Delta^{-1}(\vsd \cup \mathcal{A})$ is also $\SIG{1}{1}$. As $\WF \subseteq \Delta^{-1}(\vsd \cup \mathcal{A})$ and $\WF$ is $\PI{1}{1}$-complete, by Suslin's theorem there exists $T \in \IF$ such that $\Delta(T) \in \vsd \cup \mathcal{A}$. Because $T \in \IF$, $M_d$ is isometric to a 1-Lipschitz retract of $M_{\Delta(T)}$. Given our assumption about $\mathcal{A}$ and the fact that $\vsd$ also enjoys this property, it follows that $d \in \vsd \cup \mathcal{A}$. Therefore $\vsd \cup \mathcal{A}=\mathcal{R}$.
\end{proof}

\begin{proof}[Proof of \Cref{thm:universal_1}]
 Let $X$ be a separable Banach space. By \Cref{co:analytic}, $\langle X \rangle^{\vs}_{\hookrightarrow}$ is $\SIG{1}{1}$. Let $d' \in \vs$ and $d \in \langle X \rangle^{\vs}_{\hookrightarrow}$ such that $M_{d'}$ is isometric to a 1-Lipschitz retract of $M_d$. By \cite{weaver:18}*{Theorems 3.6 and 3.7}, $\free{M_{d'}} \hookrightarrow \free{M_d}$ and thus $d' \in \langle X \rangle^{\vs}_{\hookrightarrow}$.
 
 Assume that $\vsd \subseteq \langle X \rangle^{\vs}_{\hookrightarrow}$. By \Cref{lem:main_tool}, $\langle X \rangle^{\vs}_{\hookrightarrow} = \vs$. As $Y \hookrightarrow \free{c_0}$ whenever $Y$ is a separable Banach space (\cite{dutrieux:ferenczi:06}*{p.~1042}) and $c_0 \equiv M_d$ for some $d \in \vs$, it follows that $Y \hookrightarrow X$ for all such $Y$.
\end{proof}

\begin{proof}[Proof of \Cref{thm:universal_2}]
 Let $M$ be a complete separable metric space. By \Cref{co:analytic}, $\langle M \rangle^{\vs}_{\hookrightarrow}$ is $\SIG{1}{1}$. Assume that $\vsd \subseteq \langle M \rangle^{\vs}_{\hookrightarrow}$. By \Cref{lem:main_tool}, $\langle M \rangle^{\vs}_{\hookrightarrow} = \vs$.
\end{proof}

\begin{proof}[Proof of \Cref{thm:fail_BAP_simple}]
Using \cite{ghawadrah:17}, \cite{cuth:doucha:kurka:22a}*{Theorem 8} and \Cref{pr:reduction}, the set
\[
\vsbap :=\set{d \in \vs}{\free{M_d} \text{ has the BAP}}
\]
is $\DEL{1}{1}$. Let $d' \in \vs$ and $d \in \vsbap$ such that $M_{d'}$ is isometric to a 1-Lipschitz retract of $M_d$. Again by \cite{weaver:18}*{Theorems 3.6 and 3.7}, $\free{M_{d'}}$ is linearly isometric to a 1-complemented subspace of $\free{M_d}$ and thus $d' \in \vsbap$. For a contradiction assume $\vsd\subseteq\vsbap$. Then $\vsbap=\mathcal{R}$ by \Cref{lem:main_tool}. However, this is false because we can pick $d \in \vs$ such that $M_d$ is isometric to a separable Banach space that fails the AP; as is pointed out in the Introduction, $d \notin \vsbap$ by \cite{godefroy:kalton:03}*{Theorem 3.1}.

Now pick $d \in \vsd\setminus\vsbap$. Again, as stated in the Introduction, $\free{M_d}$ has the AP by \cite{aliaga:nous:petitjean:prochazka:21}*{Corollary 2.8}. That $\free{M_d}$ cannot be isomorphic to a dual space is a simple consequence of the classical result of Grothendieck which asserts that a separable Banach space has the MAP whenever it has the AP and is isometric to a dual space.
\end{proof}

\begin{remark}
The set $\vsd\setminus\vsbap$ is $\PI{1}{1}$ by \Cref{pr:coanalytic} and the fact that $\vsbap$ is $\DEL{1}{1}$. \Cref{thm:fail_BAP_simple} shows that $\vsd\setminus\vsbap$ is non-empty. By a simple adjustment to the argument we see moreover that $\vsd\setminus\vsbap$ cannot be $\DEL{1}{1}$: if it was $\DEL{1}{1}$ then $\vsd \cup \vsbap$ would be as well, thus it would be equal to $\vs$ by \Cref{lem:main_tool}, which again is false. We make some remarks based on this observation. 
\begin{enumerate}
\item Given metric spaces $M$ and $N$, let us write $M\stackrel{u}{\hookrightarrow}N$ if $M$ is uniformly homeomorphic to a subspace of $N$. Similarly to \cite{clemens:12}*{Lemma 4}, the set
\[
\{(d,d') \in \vs\times\vs\,:\,M_d \stackrel{u}{\hookrightarrow} M_{d'}\}
\]
is $\SIG{1}{1}$ and thus so is $\langle M \rangle^{\vs,\,u}_{\hookrightarrow}:=\{d \in \vs\,:\,M_d \stackrel{u}{\hookrightarrow} M\}$ whenever $M$ is complete and separable. Evidently $\langle M_d \rangle^{\vs,\,u}_{\hookrightarrow} \subseteq \vsd$ whenever $d\in\vsd$. Hence, given a countable subset $C \subseteq \vsd\setminus\vsbap$, the set $\Gamma:=\bigcup_{d \in C} \langle M_d \rangle^{\vs,\,u}_{\hookrightarrow} \setminus \vsbap$ is $\SIG{1}{1}$ and is included in $\vsd\setminus\vsbap$, which is $\PI{1}{1}$ but not $\DEL{1}{1}$ from above. By Suslin's theorem there exists $d' \in \vsd\setminus(\vsbap \cup \Gamma)$. We can apply transfinite induction to this fact to obtain a sequence $(d_\alpha)_{\alpha<{\omega_1}}$ of elements of $\vsd\setminus\vsbap$ such that $M_{d_\beta}\not\stackrel{u}{\hookrightarrow}M_{d_\alpha}$ whenever $\alpha < \beta < \omega_1$. I don't know if the $d_\alpha$ can be constructed so that $\free{M_{d_\alpha}}$ is not isomorphic to $\free{M_{d_\beta}}$ whenever $\alpha\neq\beta$.
\item We can replace $\vsd$ by $\vspu$ to show that the set $\vspu\setminus\vsbap$ (which is $\PI{1}{1}$ by \Cref{pr:p1u_coanalytic}) is not $\DEL{1}{1}$. No changes to the proofs above are needed. Given $d \in \vspu$, we have $\langle \free{M_d} \rangle^{\vs}_{\hookrightarrow} \subseteq \vspu$ by \cite{aliaga:gartand:petitjean:prochazka:22}*{Theorem 4.6}. Therefore, by arguing as above, there is a sequence $(d_\alpha)_{\alpha<\omega_1}$ of elements of $\vspu\setminus\vsbap$ such that $\free{M_{d_\beta}}\not\hookrightarrow\free{M_{d_\alpha}}$ whenever $\alpha<\beta<\omega_1$. I don't know if the spaces $\free{M_{d_\alpha}}$ have the AP.
\item While the above results show that there are many examples of complete separable p1u spaces whose free spaces fail the BAP, they don't provide an explicit description of any of them. It could be instructive to obtain concrete examples.
\end{enumerate}
\end{remark}

I am grateful to G. Lancien for suggesting the following problem which \Cref{thm:fail_BAP_simple} leaves open.

\begin{problem}\label{prob:embed_in_separable_dual}
 Let $M$ be a complete separable p1u metric space. Does there exist a separable Banach space $X$ such that $\free{M} \hookrightarrow X^*$? 
\end{problem}

We finish this section with two direct applications of \Cref{lem:combine}.

\begin{proposition}\label{pr:complete}
 The sets $\vsd$ and $\vspu$ are $\PI{1}{1}$-complete.
\end{proposition}

\begin{proof}
Recall that, by Propositions \ref{pr:coanalytic} and \ref{pr:p1u_coanalytic}, these sets are $\PI{1}{1}$. Select any complete separable non-p1u metric space $M$. Then the map $\Delta$ furnished by \Cref{lem:combine} applied to $M$ witnesses the fact that $\WF$ is reducible to $\vsd$ and $\vspu$.
\end{proof}

In the final result of this section, we obtain a lower bound on the complexity of the set of metrics in $\vs$ whose free spaces satisfy the AP. The complexity calculation is completed in \Cref{sect:ap}.

\begin{proposition}\label{pr:AP_PI-hard}
 The set $\vsap:=\set{d \in \vs}{\free{M_d} \text{ has the AP}}$ is $\PI{1}{1}$-hard.
\end{proposition}

\begin{proof}
Apply \Cref{lem:combine} to a separable Banach space $M$ that fails the AP. If $T\in \WF$ then $\free{M_{\Delta(T)}}$ has the AP again by \cite{aliaga:nous:petitjean:prochazka:21}*{Corollary 2.8}. If $T \in \IF$ then $M$ is isometric to a 1-Lipschitz retract of $M_{\Delta(T)}$. Once again by \cite{weaver:18}*{Theorems 3.6 and 3.7}, $\free{M}$ is linearly isometric to a 1-complemented subspace of $\free{M_{\Delta(T)}}$. Since $\free{M}$ fails the AP, so must $\free{M_{\Delta(T)}}$. Therefore $\Delta$ witnesses the fact that $\WF$ is reducible to $\vsap$.
\end{proof}

\section{Complexity and the approximation property}\label{sect:ap}

This section is solely concerned with determining the complexity of two classes of Banach spaces having the AP. The proof of the result below is an adaptation of an argument from \cite{ghawadrah:17}. As above, we frame the result in terms of the Polish space $\prsb$ instead of the Effros-Borel space $\se$.

\begin{theorem}\label{thm:ap_coanalytic}
The sets $\prsbap:= \set{\mu \in \prsb}{X_\mu \text{ has the AP}}$ and $\vsap$ are $\PI{1}{1}$-complete.
\end{theorem}

\begin{proof}
Having in mind Propositions \ref{pr:reduction} and \ref{pr:AP_PI-hard}, it is enough to prove that $\prsbap$ is $\PI{1}{1}$. Let $(e_i)_{i\in\omega} \subseteq V$ denote the standard basis of $c_{00}$; then $(e_i)_{i\in\omega}$ is a linearly independent sequence having dense span in $X_\mu$ whenever $\mu \in \prsb$. Given $R \in \omega$, let $V_R=\Q$-$\aspan(e_i)_{i<R} \subseteq V$. Define the closed subset
\[
 RC = \set{(\mu,x) \in \prsb\times V^{\omega\times\omega}}{\mu(x(k,p)-x(k,q)) \leq 2^{-p} \text{ for all }k,p,q\in\omega,\,q \geq p}
\]
of the Polish space $\prsb\times V^{\omega\times\omega}$ (where $V$ has the discrete topology). Then, given finite $F \subseteq \omega$ and $n,k,j \in \omega$, define the closed subset
\[
 TB_{n,F,k,j} = \set{(\mu,x) \in \prsb\times V^{\omega\times\omega}}{\mu(x(k,q)-x(j,q))\leq 2^{-n} \text{ for all }q,n \in \omega,\,q \geq n+1}
\]
of $\prsb\times V^{\omega\times\omega}$ and the $\PI{0}{3}$ subset
\[
 TB = \bigcap_{n \in \omega}\bigcup_{\substack{F \subseteq \omega \\ \mathrm{finite}}}\bigcap_{k\in\omega}\bigcup_{j \in F} TB_{n,F,k,j}.
\]
Next, given $n,L,R,N \in \omega$, $v = (v_i)_{i<N} \in V_R^N$ and $\alpha=(\alpha_i)_{i<N} \in \Q^N$, define the closed subset
\[
FR_{L,N,v,\alpha} = \set{\mu \in \prsb}{\mu\bigg(\sum_{i<N} \alpha_i v_i\bigg) \leq L\mu\bigg(\sum_{i<N} \alpha_i e_i\bigg)}
\]
of $\prsb$, and the closed subset
\begin{align*}
 AP_{n,L,N,v} =& \bigg\{(\mu,x)\in \prsb\times V^{\omega\times\omega}\,:\, q \leq N \text{ and }x(k,q) \in V_N \text{ implies}\\
 & \mu\bigg(x(k,q) - \sum_{i<N}x(k,q)(i)v_i\bigg) \leq 2^{-n}+(L+2)\cdot 2^{-q}\bigg\}
\end{align*}
of $\prsb\times V^{\omega\times\omega}$. Then define the $\PI{0}{5}$ subset
\[
 AP = \bigcap_{n\in\omega}\bigcup_{L,R\in\omega}\bigcap_{N \in \omega}\bigcup_{v\in V_R^N}\bigg(AP_{n,L,N,v} \cap \bigcap_{\alpha \in \Q^N} (FR_{L,N,v,\alpha} \times V^{\omega\times\omega})\bigg).
\]
of $\prsb\times V^{\omega\times\omega}$. Now let $\mathcal{A} \subseteq \prsb$ be the $\SIG{1}{1}$ image of the projection onto the first coordinate of $(RC \cap TB)\setminus AP \subseteq \prsb \times V^{\omega\times\omega}$. We show that $\prsbap = \prsb\setminus \mathcal{A}$ and thus $\prsbap$ is $\PI{1}{1}$.

Let $\mu \in \prsb\setminus \mathcal{A}$. We prove that $X_\mu$ has the AP, i.e.~$\mu\in\prsbap$. Let $K \subseteq X_\mu$ be compact and let $\ep>0$. We can assume that $\mu(z) \leq 1$ whenever $z \in K$. Let $(z_k)_{k\in\omega}$ be a dense sequence in $K$. Then pick $x \in V^{\omega\times\omega}$ such that $\mu(x(k,q)-z_k)\leq 2^{-q-1}$ for all $k,q\in\omega$. We verify that $(\mu,x) \in RC \cap TB$. Evidently $(\mu,x) \in RC$. Now let $n \in \omega$. Because $(z_k)_{k\in\omega}$ is totally bounded, there is finite $F \subseteq \omega$ such that whenever $k\in\omega$, there exists $j \in F$ satisfying $\mu(z_k-z_j) \leq 2^{-n-1}$. By the choice of $x$,
\[
\mu(x(k,q)-x(j,q)) \leq 2^{-q} + 2^{-n-1} \leq 2^{-n}, \quad q \geq n+1.
\]
It follows that $(\mu,x) \in TB$.

As $\mu \notin \mathcal{A}$ and $(\mu,x) \in RC \cap TB$, it must be that $(\mu,x) \in AP$. Fix $n\in\omega$ such that $2^{-n} \leq \ep$. Given $N \in \omega$, by equivalence of norms on finite-dimensional spaces, there exists $\theta_N > 0$ such that
\begin{equation}\label{eqn:equivalent_norms}
\sum_{i<N} |\alpha_i| \leq \theta_N\mu\bigg(\sum_{i<N} \alpha_i e_i\bigg),\quad (\alpha_i)_{i<N} \in \R^N.
\end{equation}
As $(\mu,x) \in AP$, there exist $L,R \in \omega$ and $v^{(N)} \in V_R^N$, $N \in \omega$, such that
\[
(\mu,x) \in AP_{n,L,N,v^{(N)}} \cap \bigcap_{\alpha \in \Q^N} (FR_{L,N,v^{(N)},\alpha} \times V^{\omega\times\omega}), \quad N \in \omega.
\]
Given $\alpha \in \Q^{N}$, as $\mu \in FR_{L,N,v^{(N)},\alpha}$,
\begin{equation}\label{eqn:T_bounded}
\mu\bigg(\sum_{i<N} \alpha_i v_i^{(N)}\bigg) \leq L\mu\bigg(\sum_{i<N} \alpha_i e_i\bigg).
\end{equation}
In particular, $\mu(v^{(N)}_i) \leq L\mu(e_i)$ whenever $i<N$. Since each $(v_i^{(N)})$, $i\in\omega$, is a bounded sequence in $V_R \subseteq \aspan(e_j)_{j<R}$, by a diagonalisation argument there are $w_i \in \aspan(e_j)_{j < R}$ and $N_0 < N_1 < N_2 < \ldots$ such that $\mu\big(v^{(N_\ell)}_i-w_i\big)\to 0$ for all $i\in\omega$. Using \eqref{eqn:T_bounded} and the continuity of $\mu$, given $\alpha \in \Q^N$,
\[
\mu\bigg(\sum_{i<N} \alpha_i w_i\bigg) \leq L\mu\bigg(\sum_{i<N} \alpha_i e_i\bigg).
\]
Hence we can define a bounded $\Q$-linear operator $\map{T}{V}{\aspan(e_i)_{i<R}}$ by setting $Te_i = w_i$, $i \in \omega$, and then extending to finite $\Q$-linear combinations of the $e_i$. By uniform continuity we extend it again to a bounded linear operator on $X_\mu$ satisfying $TX_\mu \subseteq \aspan(e_i)_{i<R}$ and $\mu(T) \leq L$. 

By taking a further subsequence of natural numbers, also labelled $(N_\ell)$, we can ensure that
\begin{equation}\label{eqn:mu_estimate}
\mu\big(v_i^{(N_\ell)}-w_i\big) \leq \frac{2^{-\ell-1}}{\theta_\ell}, \quad i < \ell.
\end{equation}
Now let $k,q \in \omega$. Pick $\ell\in\omega$ large enough so that $q \leq \ell$ and $x(k,q) \in V_\ell$. Then $\ell \leq N_\ell$ implies $q \leq N_\ell$ and $x(k,q) \in V_{N_\ell}$. Because $(\mu,x) \in AP_{n,L,N_\ell,v^{(N_\ell)}}$, 
\begin{equation}\label{eqn:AP_1}
\mu\bigg(x(k,q) - \sum_{i<N_\ell}x(k,q)(i)v_i^{(N_\ell)}\bigg) \leq 2^{-n}+(L+2)\cdot 2^{-q}.
\end{equation}
We estimate
\begin{align*}
\mu\bigg(\sum_{i<\ell} x(k,q)(i)\big(v_i^{(N_\ell)}-w_i\big) \bigg) &\leq \sum_{i<\ell} |x(k,q)(i)|\mu\big(v_i^{(N_\ell)}-w_i\big)\\
&\leq \frac{2^{-\ell-1}}{\theta_\ell}\sum_{i<\ell}|x(k,q)(i)| \tag*{by \eqref{eqn:mu_estimate}}\\
&\leq 2^{-\ell-1}\mu\bigg(\sum_{i<\ell}x(k,q)(i)e_i \bigg) \tag*{by \eqref{eqn:equivalent_norms}}\\
&= 2^{-\ell-1}\mu(x(k,q)) \tag*{as $x(k,q) \in V_\ell$}\\
&\leq 2^{-\ell} \leq 2^{-q}.
\end{align*}
The penultimate inequality above holds because $\mu(x(k,q)) \leq \mu(z_k) + 2^{-1} \leq \frac{3}{2}$ for all $k,q \in \omega$.
Therefore
\begin{align*}
\mu(x(k,q)-Tx(k,q)) &= \mu\bigg(x(k,q)-\sum_{i<\ell}x(k,q)(i)w_i \bigg) \tag*{as $x(k,q) \in V_\ell$}\\
&\leq \mu\bigg(x(k,q)-\sum_{i<\ell}x(k,q)(i)v_i^{(N_\ell)} \bigg) + 2^{-q}\\
&= \mu\bigg(x(k,q)-\sum_{i<N_\ell}x(k,q)(i)v_i^{(N_\ell)} \bigg) + 2^{-q} \tag*{as $\ell \leq N_\ell$}\\
&\leq 2^{-n} + (L+3)\cdot 2^{-q} \tag*{by \eqref{eqn:AP_1}.}
\end{align*}
Fixing $k\in\omega$ and letting $q \to \infty$ yields $\mu(z_k - Tz_k) \leq 2^{-n}$. Since this holds for all $k\in\omega$, by continuity $\mu(z-Tz) \leq 2^{-n} \leq \ep$ for all $z \in K$. Therefore $X_\mu$ has the AP.

Conversely, assume that $\mu\in\prsbap$. Let $(\mu,x) \in RC \cap TB$. We show that $(\mu,x) \in AP$ and hence $\mu \in \prsb\setminus \mathcal{A}$. Let $n \in \omega$. Given $k\in\omega$, let $z_k = \lim_q x(k,q) \in X_\mu$. As $(\mu,x) \in TB$ as well, it is easy to see that $\set{z_k}{k\in\omega}$ is totally bounded and thus
\[
K:= \cl{\set{z_k}{k\in\omega}} \subseteq X_\mu
\]
is compact. As $X_\mu$ has the AP, there exists a bounded linear finite-rank operator $\map{T}{X_\mu}{X_\mu}$ such that $\mu(z-Tz) < 2^{-n}$ for all $z\in K$. Fix $L \in \omega$ such that $\mu(T)<L$. By a perturbation argument, we can assume there exists $R\in\omega$ such that $TX_\mu \subseteq \aspan(e_i)_{i<R}$. Let $B > 0$ such that $\mu(z)\leq B$ for all $z \in K$.  Given $N\in\omega$, pick $\theta_N>0$ to satisfy \eqref{eqn:equivalent_norms} above and choose $v \in V^N_R$ such that 
\begin{equation}\label{eqn:mu_estimate_2}
\mu(v_i-Te_i) \leq \frac{1}{\theta_N}\min\bigg\{\frac{2^{-N}}{B+\frac{1}{2}},L-\mu(T) \bigg\}, \quad i<N.
\end{equation}
Then, given $\alpha \in \Q^N$,
\begin{align*}
\mu\bigg(\sum_{i<N} \alpha_i v_i \bigg) &\leq \mu\bigg(\sum_{i<N} \alpha_i Te_i\bigg) + \sum_{i<N} |\alpha_i|\mu(v_i-Te_i)\\
&\leq \mu(T)\mu\bigg(\sum_{i<N}\alpha_i e_i\bigg) + (L-\mu(T))\mu\bigg(\sum_{i<N}\alpha_i e_i\bigg) \tag*{by \eqref{eqn:equivalent_norms}, \eqref{eqn:mu_estimate_2}}\\
&= L\bigg(\sum_{i<N}\alpha_i e_i\bigg).
\end{align*}
We conclude that
\[
(\mu,x) \in \bigcap_{\alpha \in \Q^N} (FR_{L,N,v,\alpha} \times V^{\omega \times \omega}).
\]

We also need to show that $(\mu,x) \in AP_{n,L,N,v}$. Assume that $q \leq N$ and $x(k,q) \in V_N$. As $(\mu,x) \in RC$ and $\mu(T) \leq L$, we have
\begin{equation}\label{eqn:AP_2}
\mu(x(k,q)-z_k) \leq 2^{-q} \quad\text{and}\quad \mu(Tx(k,q)-Tz_k) \leq L\cdot 2^{-q}.
\end{equation}
Moreover, 
\begin{align*}
\mu\bigg(Tx(k,q) - \sum_{i<N} x(k,q)(i)v_i\bigg) &\leq \sum_{i<N} |x(k,q)(i)|\mu(Te_i-v_i) \tag*{as $x(k,q) \in V_N$}\\
&\leq \frac{2^{-N}}{B+\frac{1}{2}}\mu\bigg(\sum_{i<N} x(k,q)(i)e_i\bigg) \tag*{by \eqref{eqn:equivalent_norms}, \eqref{eqn:mu_estimate_2}}\\
&= \frac{2^{-N}}{B+\frac{1}{2}}\mu(x(k,q))\\
&\leq 2^{-N} \leq 2^{-q}.
\end{align*}
The penultimate inequality above holds as $\mu(x(k,q)) \leq \mu(z_k) + 2^{-1} \leq B+\frac{1}{2}$. Given that $\mu(z_k-Tz_k) \leq 2^{-n}$, by \eqref{eqn:AP_2} and the above we conclude that
\[
\mu\bigg(x(k,q) - \sum_{i<N} x(k,q)(i)v_i\bigg) \leq 2^{-n} + (L+2)\cdot 2^{-q}.
\]
Therefore $(\mu,x) \in AP_{n,L,N,v}$. The proof is complete.
\end{proof}

\bibliography{free_bossard}

\end{document}